\documentclass{amsart}
\usepackage{tabularx}
\usepackage{enumitem,amsmath,amsfonts,amssymb,xcolor}
\usepackage{latexsym,amsfonts,euscript,amssymb,gensymb}
\newtheorem{lemma}{\bf Lemma}[section]

\newtheorem{defi}[lemma]{\bf Definition}

\newtheorem{thm}[lemma]{\bf Theorem}

\newcommand{\GL}{{\operatorname{GL}}}
\newcommand{\SL}{{\operatorname{SL}}}
\newcommand{\PGL}{{\operatorname{PGL}}}

\newcommand{\PSL}{{\operatorname{PSL}}}
\newcommand{\GU}{{\operatorname{GU}}}
\newcommand{\SU}{{\operatorname{SU}}}
\newcommand{\PGU}{{\operatorname{PGU}}}
\newcommand{\PSU}{{\operatorname{PSU}}}

\newcommand{\OO}{{\operatorname{P\Omega}}}

\DeclareMathOperator{\Sym}{Sym}
\DeclareMathOperator{\Aut}{Aut}
\DeclareMathOperator{\Out}{Out}
\DeclareMathOperator{\soc}{soc}

\DeclareMathOperator{\Sz}{Sz}
\DeclareMathOperator{\Ree}{Ree}
\DeclareMathOperator{\STS}{STS}
\DeclareMathOperator{\KTS}{KTS}

\input xy
\xyoption{all}

\title[]{The structure of $3$-pyramidal groups}

\author{Xiaofang Gao} 
\address{Xiaofang Gao. Departamento de Matem\'atica, Universidade de Bras\'ilia, Campus 
Universit\'ario \\ Darcy Ribeiro, Bras\'ilia-DF, 70910-900, Brazil. \newline
ORCID:  https://orcid.org/0000-0001-8106-7941}
\email{gaoxiaofang2020@hotmail.com}

\author{Martino Garonzi}
\address{Martino Garonzi. Departamento de Matem\'atica, Universidade de Bras\'ilia, Campus 
Universit\'ario Darcy Ribeiro, Bras\'ilia-DF, 70910-900, Brazil. \newline
ORCID: https://orcid.org/0000-0003-0041-3131}
\email{martino@mat.unb.br, mgaronzi@gmail.com}

\thanks{The first author acknowledges the support of the CAPES PhD
fellowship and the NSF of China - Grant number 12161035. 
The second author acknowledges the support of Conselho Nacional 
de Desenvolvimento Cient\'ifico e Tecnol\'ogico (CNPq), Universal
- Grant number 402934/2021-0.}

\date{}


\begin{document}

\setlength{\parskip}{2mm}

\maketitle

\begin{abstract}
A combinatorial block design $D$ is called $3$-pyramidal if there exists a subgroup $G$ of $\Aut(D)$ fixing $3$ points and acting regularly on the other points. If this happens, we say that the design is $3$-pyramidal under $G$. In case $D$ is a Kirkman triple system, it is known that such a group $G$ has precisely $3$ involutions, all conjugate to each other. In this paper, we obtain a classification of the groups with this property.
\end{abstract}

\section{Introduction}

 A Steiner triple system of order $v$, briefly $\STS(v)$, is a pair $(V,B)$ where $V$ is
a set of $v$ points and $B$ is a set of $3$-subsets (blocks) of $V$ with the property that any two
distinct points are contained in exactly one block. A Kirkman triple system of order $v$, briefly
$\KTS(v)$, is an $\STS(v)$ together with a resolution $R$ of its block-set $B$, that is a partition of $B$
into classes (parallel classes) each of which is, in its turn, a partition of the point-set $V$.  It is known that a $\KTS(v)$ exists if and only if $v\equiv3 \mod 6$. These structures fall into the broader category of combinatorial designs. A combinatorial design $D$ is said to be $3$-pyramidal if there exists a subgroup of $\Aut(D)$ which fixes $3$ points and acts sharply transitively on all the other points.  Bonvicini, Buratti,  Garonzi,  Rinaldi and 
Traetta \cite{SMG} gave a necessary condition for the existence of a $3$-pyramidal $\KTS(v)$, provided some difference methods to construct $3$-pyramidal Kirkman Triple Systems and observed that each
group having a $3$-pyramidal action on a Kirkman Triple System must have exactly three involutions, and
these involutions are pairwise conjugate (see \cite[Theorem 3.8]{SMG}). In this paper, we are concerned with studying groups with this property. 

All groups considered in this paper will be finite. An involution is an element of order $2$.

\begin{defi}[3-pyramidal group]
We say that a finite group $G$ is $3$-pyramidal if it has exactly $3$ involutions, which are all conjugate to each other.
\end{defi}

In \cite{SMG}, $3$-pyramidal groups are called pertinent. Easy examples of $3$-pyramidal groups are the symmetric group $S_3$ and the alternating group $A_4$. Observe that a $3$-pyramidal group is necessarily nonabelian of order divisible by $6$, since it contains involutions and it acts transitively on a set of size $3$. Moreover, if $G$ is $3$-pyramidal and $H$ is any group of odd order, then $G \times H$ is also $3$-pyramidal.

There is literature about finite solvable groups in which all the involutions are conjugate, see for example \cite[Section 8 of Chapter IX]{HB}. The Sylow $2$-subgroups of such groups were classified by Shaw \cite{shaw}, they are cyclic, generalized quaternion (in case there is only one involution), homocyclic (direct products of isomorphic cyclic groups) or Suzuki $2$-groups, i.e. $2$-groups $P$ that admit a solvable subgroup of $\Aut(P)$ whose action on the involutions of $P$ is transitive. We refer to \cite[Section 7 of Chapter VIII]{HB} for the structure of Suzuki $2$-groups.

Our main theorem is the following.



\begin{thm} \label{main}
Let $G$ be a finite group and $O(G)$ the largest normal subgroup of $G$ of odd order. Let $K$ be the subgroup generated by the involutions of $G$. Then $G$ is $3$-pyramidal if and only if one of the following holds.

\begin{enumerate}
\item $G$ is isomorphic to $S_3\times H$ where $H$ is a group of odd order.
\item $O(G) \leqslant C_G(K)$ and $G/O(G)$ is isomorphic to $N \rtimes A$ where $N$ is the Suzuki $2$-group of order $64$ and $A$ is a subgroup of $\Aut(N)$ of order $3$ or $15$.
\item $O(G) \leqslant C_G(K)$ and $G/O(G)$ is isomorphic to $(C_{2^n} \times C_{2^n}) \rtimes A$ where $A$ is the cyclic group of order $3$ generated by the automorphism $(a,b) \mapsto (b,(ab)^{-1})$.
\end{enumerate}
In item (1) $K \cong S_3$ while in items (2), (3) $K \cong C_2 \times C_2$.
\end{thm}

By the Feit-Thompson theorem, all finite groups of odd order are solvable. Therefore, the above theorem implies that all $3$-pyramidal groups are solvable. The quotients $G/O(G)$ in item (2) are SmallGroup(192,1025), SmallGroup(960, 5748).

We can construct infinite families of $3$-pyramidal groups as follows: let $Y$ be a group of odd order with a normal subgroup $X$ of index $3$ and let $N := C_{2^n} \times C_{2^n}$. The group $Y$ acts on $N$ as an automorphism of order $3$ as in item (3) of the statement of Theorem \ref{main}, by composition $Y \to Y/X \cong C_3 \to \Aut(N)$. The semidirect product $G:=N \rtimes Y$ is $3$-pyramidal and $O(G)=X$. Also, note that there exist $3$-pyramidal groups whose Sylow $2$-subgroups are not normal. An example is SmallGroup(1296,2705).

Our paper is organized as follows. In Section \ref{lemmas} we give some definitions and lemmas. In Section \ref{KTSsolvable}
we prove that $3$-pyramidal groups are solvable using the structure of primitive groups and the classification of finite simple groups. In Section \ref{proofmain} we finish the proof of Theorem \ref{main} with the help of a theorem of Thompson and the structure of Suzuki $2$-groups.

\section{Preliminary results} \label{lemmas}

The following lemma summarizes some properties of $3$-pyramidal groups.
\begin{lemma} \label{propertiesKTS}
Let $G$ be a $3$-pyramidal group and write $|G|=2^n \cdot d$ with $d$ odd. Let $H$ be a subgroup of $G$, let $K$ be the subgroup of $G$ generated by the $3$ involutions and $C:=C_G(K)$.
\begin{enumerate}[label=(\roman*)]
\item If $H$ has even order and $HC = G$, then $H$ is $3$-pyramidal.
\item Suppose the involutions of $G$ commute pairwise. If $H$ is normal in $G$ and $|H|$ is odd
then $G/H$ is a $3$-pyramidal group.
\item If $n=1$ then $K \cong S_3$ and $G \cong C \times K$.
\item If $n \geqslant 2$ then $n$ is even, $K \cong C_2 \times C_2$ and $|G:C|=3$.
\item If $H$ is normal in $G$ and $|H|=2^m$ then $m$ is even.
\end{enumerate}
\end{lemma}

\begin{proof}
Almost everything is part of \cite[Section 4]{SMG} (where $3$-pyramidal groups are called pertinent) including the fact that $G/C \cong S_3$ if $n=1$. In particular, the fact that if $n \geqslant 2$ then $n$ is even is the content of \cite[Theorem 4.6]{SMG}. The only thing we need to explain here is the fact that, if $n=1$, then $K \cong S_3$ and $G \cong C \times K$. Let $i,j,k$ be the three involutions of $G$. Note that, if $n=1$, then $ij \neq ji$, since two commuting involutions generate a subgroup of order $4$. So $i^j \neq i$ and of course $i^j \neq j$ being $i \neq j$, therefore $i^j = k$. The same argument shows that $j^i=k$, therefore $ij$ has order $3$ and $K \cong S_3$. Since $C \cap K = Z(K) = \{1\}$ and $G/C \cong S_3$, we have $|CK| = |C| \cdot |K| = |G|$ hence $CK = G$ and, since $C,K$ are normal in $G$, it follows that $G \cong C \times K$.
\end{proof}

The above lemma implies in particular that a $3$-pyramidal group generated by involutions must be isomorphic to $S_3$.

We include a result about almost-simple groups that will be useful in the proof. An almost-simple group with socle $S$ is a group $X$ which admits a unique minimal normal subgroup, call it $S$, which is nonabelian and simple. If this is the case, then $C_X(S)=\{1\}$ hence we may identify $X$ with a subgroup of $\Aut(S)$. The maximal subgroup $H$ of $S$ is $X$-ordinary if its $X$-class equals its $S$-class, in other words for every $x\in X$ there exists $s\in S$ such that $H^x=H^s$. Note that, if $S \leqslant X \leqslant Y \leqslant \Aut(S)$ and $H$ is a $Y$-ordinary maximal subgroup of $S$, then $H$ is also $X$-ordinary. We say that $H$ is ordinary if it is $\Aut(S)$-ordinary, so that it is $X$-ordinary for every almost-simple group $X$ with socle $S$.

The $\Aut(S)$-class of $H$
is a disjoint union of $S$-classes: there exist $\varphi_1, \ldots, \varphi_c \in \Aut(S)$
such that the subgroups $H^{\varphi_i}\leqslant S$ are pairwise not conjugate in $S$ and for
every $\varphi \in \Aut(S)$ there exist $s\in S$, $i\in\{1, \cdots, c\}$ such that
$H^\varphi=H^{\varphi_i s}$. Let $C_i = \{H^{\varphi_i s} : s\in S\}$ be the $S$-class of $H^{\varphi_i}$, for $i=1, \ldots, c$. 
Then $\Aut(S)$ acts transitively on the set $\{C_i\ :\ i= 1, \ldots, c\}$ by sending $(C_i, \varphi)$ to the
$S$-class of $H^{\varphi_i\varphi}$. This corresponds to a homomorphism $\pi: \Aut(S) \to \Sym(c)$. Note that $H$ is $\Aut(S)$-ordinary if and only if $c=1$.

\begin{lemma} \label{X-ord}
Let $X$ be an almost-simple group with socle $S$, $H$ a maximal
subgroup of $S$ and $\pi$ the map defined above. 
\begin{enumerate}
\item If $X \leqslant \ker(\pi)$ then $H$ is $X$-ordinary.
\item If $X$ is normal in $\Aut(S)$ and $H$ is $X$-ordinary, then $X \leqslant \ker(\pi)$.
\item If $c=2$ and $X$ does not have $C_2$ as a quotient, then $H$ is $X$-ordinary.
\item If $H$ is $X$-ordinary, then $N_X(H)$ is a maximal subgroup of $X$ and the intersection $N_X(H)\cap S$ is equal to $H$.
\end{enumerate}
\end{lemma}

\begin{proof}
Item 1. Assume $X \leqslant \ker(\pi)$ and let $x \in X$. Then $H^x$ belongs to the $S$-class of $H$,
so there is some $s \in S$ with $H^x=H^s$. This shows that $H$ is $X$-ordinary. 

Item 2. Assume that $X$ is normal in $\Aut(S)$, that $H$ is $X$-ordinary and
let $x \in X$. Then there is some $s\in S$ with $H^x=H^s$. If $i\in \{1, \ldots, c\}$ then, since
$\varphi_i x {\varphi_i}^{-1}\in X$, being $X \unlhd \Aut(S)$, there is some $s \in S$ with
$H^{\varphi_i x {\varphi_i}^{-1}}=H^s$, therefore $H^{\varphi_i x} = H^{s \varphi_i} = H^{\varphi_i {\varphi_i}^{-1} s \varphi_i}$
belongs to the $S$-class of $H^{\varphi_i}$ since $S \unlhd \Aut(S)$.

Item 3. Let $N:=\ker(\pi)$. Then
$$X/X \cap N\cong XN/N \leqslant \Aut(S)/N\cong \mbox{Im}(\pi) \leqslant \Sym(c).$$
In particular, if $c=2$ then the index $|X:X\cap N|$ is $1$ or $2$. So if $c=2$ and $X$ does not
have $C_2$ as a quotient then $|X:X\cap N|=1$, i.e. $X \leqslant N = \ker(\pi)$. This implies that $H$ is $X$-ordinary by item (1).

Item 4. Since $H$ is $X$-ordinary, for every $x\in X$, there is $s\in S$ such that
$H^x=H^s$, hence $H^{xs^{-1}}=H$,
and hence $xs^{-1}\in N_X(H)$. Therefore $N_X(H)S = X$.
Since $H$ is not normal in $X$, $X$ has a maximal
subgroup $M$ with $N_X(H)\leqslant M<X$, so that $MS=X$. Since $M\cap S \unlhd M$, the maximality
of $M$ in $X$ implies that $N_X(M\cap S)$ is equal to one of $M,X$.
Since $\{1\} \neq H \leqslant M \cap S < S$ and $S$ is a simple group, $M \cap S$ cannot be
normal in $X$, so $N_X(M\cap S) = M$. Since $H \leqslant M \cap S < S$, the maximality of $H$ 
in $S$ implies that $M \cap S =H$.
\end{proof}

\begin{lemma} \label{G0=S}
Let $X$ be any almost-simple group with socle $S$. Assume that $X/S$ is a (possibly trivial) cyclic $3$-group and that every maximal subgroup of $X$ not containing $S$ is solvable. Let $R$ be a maximal subgroup of $X$ such that $RS=X$. Then $R \cap S$ is a maximal subgroup of $S$.
\end{lemma}
\begin{proof}
Let $Y$ be a normal subgroup of $X$ containing $S$, assume that $R \cap Y$ is a maximal subgroup of $Y$ and that $Y$ is of minimal order with these properties. This makes sense since $R$ is maximal in $X$. We aim to prove that $Y=S$. Both $Y/S$ and $X/S$ are (possibly trivial) cyclic $3$-groups. The pair $(Y, R \cap Y)$ is listed in \cite[Tables 14-20]{LZ} with the notation $(G_0,H_0)$ and we will refer to this table (and its notation) in our discussion. If $S$ is an alternating or sporadic group, then $X/S$ is a $2$-group hence $S=Y=X$. In the case of linear groups, $S=\PSL(n,q)$ with $q=p^f$, $p$ a prime, the inverse-transpose automorphism $\tau$ has order $2$ and, if $\sigma$ is a field automorphism and $\delta$ is a diagonal automorphism (induced by conjugation via an element of $\PGL(n,q)$), then $\sigma \tau \neq 1$ has even order dividing $2f$ and $(\delta \tau)^2$ is inner. In the symplectic or exceptional case, the graph automorphism $\mu$ defined in \cite[Remark 10.4]{LZ} is a $2$-element and we have $|^2F_4(2):{^2}F_4(2)'|=2$. This implies that, if $S$ is exceptional, then $Y/S$ is a $2$-group hence $Y=S$. We are left to examine the following list of groups. When dealing with small groups, we will use the Atlas \cite{CCNPW}. In the case of orthogonal groups, $\sigma$ is a field automorphism and $\tau$ is the triality of $\OO_8^+(q)$.

\begin{itemize}
\item  $Y \cong \PGL_3(4)$, $S\cong \PSL_3(4)$, $\Out(S) \cong D_{12}$. Since $X/S$ is a cyclic $3$-group, $X=Y$ has a nonsolvable maximal subgroup $2^4:(3\times A_5)$ supplementing the socle, contradicting the assumption.

\item $Y \cong \PGU_3(5)$, $S \cong \PSU_3(5)$, $\Out(S) \cong S_3$. Since $X/S$ is a cyclic $3$-group, $X=Y$ has a nonsolvable maximal subgroup $2S_5\times 3$ supplementing the socle, contradicting the assumption.

\item $Y \cong \OO_8^+(2).\langle \tau\rangle$, $S \cong \OO_8^+(2)$, $\Out(S)\cong S_3$. Since $X/S$ is a cyclic $3$-group, $X=Y$ has a nonsolvable maximal subgroup $\PSU_3(3):2\times 3$ supplementing the socle, contradicting the assumption.

\item $Y \cong \OO_8^+(3).\langle \tau\rangle$, $S\cong \OO_8^+(3)$, $\Out(S)\cong S_4$. Since $X/S$ is a cyclic $3$-group, $X=Y$ has a nonsolvable maximal subgroup $2^{3+6}(\PSL_3(2)\times 3)$ supplementing the socle, contradicting the assumption.

\item $Y = \OO_8^+(q).\langle \tau\sigma\rangle$, $S \cong \OO_8^+(q)$ where $q\neq 3$ and $o(\sigma)=3^k$, $k\geqslant 0$.
If $q=2$, then $\Out(S)\cong S_3$ and the field automorphism $\sigma$ is trivial,
this implies that $Y = \OO_8^+(2).\langle \tau\rangle$, which was considered above.
If $q>3$ then, by \cite[Table 8.50]{BHR} (see also \cite[Table I]{Kleidman}), $S$ has a nonsolvable ordinary maximal
subgroup $M$ in the Aschbacher class $\mathcal{C}_1$ of type
$[q^9](\frac{1}{d} \GL_2(q) \times \Omega_4^+(q)).d$, where $d=(q-1,2)$. By Lemma \ref{X-ord}, we have
$N_{X}(M)S=X$ and $N_{X}(M)$ is a nonsolvable maximal subgroup of $X$. This contradicts the assumption.
\end{itemize}
This concludes the proof.
\end{proof}

\section{Solvability of $3$-pyramidal groups} \label{KTSsolvable}

In this section, we will discuss the solvability of $3$-pyramidal groups and use the classification theorem for finite simple groups to prove that $3$-pyramidal groups are solvable. We will also use, without mentioning it in all cases, the Feit-Thompson theorem, which says that all finite groups of odd order are solvable.

We prove the result by contradiction. Assume $W$ is a nonsolvable $3$-pyramidal group of minimal order. Let $K$ be the subgroup of $W$ generated by the three involutions and let $C:=C_W(K)$. Note that both $K$ and $C$ are normal in $W$. An important observation, which we will use many times without special mention, is that, if $H$ is a proper subgroup of $W$ such that $HC=W$, then $H$ solvable. Indeed, either $|H|$ is odd, in which case $H$ is solvable by the Feit-Thompson theorem, or $|H|$ is even, in which case $H$ is $3$-pyramidal by Lemma \ref{propertiesKTS}(i), hence it is solvable by minimality of $W$. Write $|W| = 2^n \cdot d$, where $d$ is odd. By \cite[Exercise 1.6.19]{DJS} we have $n\geqslant 2$. In the light of Lemma \ref{propertiesKTS}(iv), $n$ is even, $K \cong C_2\times C_2$ and $W/C\cong A_3$, so $C$ is a maximal normal subgroup of $W$. If $H$ is a nontrivial normal subgroup of $W$ with odd order, Lemma \ref{propertiesKTS}(ii) implies that $W/H$ is $3$-pyramidal. Since $|W/H|<|W|$ and $|H|$ is odd, both $H$ and $W/H$ are solvable, this contradicts the nonsolvability of $W$, so $O(W)=\{1\}$. In particular, the minimal normal subgroups of $W$ have even order, so they contain all the involutions. This implies that the unique minimal normal subgroup of $W$ is $K$.

We divide the proof into several steps. Recall that a chief factor of $W$ is a minimal normal subgroup of a quotient of $W$. Since a minimal normal subgroup is always characteristically simple, it is isomorphic to a direct power of a simple group. As usual, $\Phi(W)$ denotes the Frattini subgroup of $W$.

\subsection{Step 1} We will prove that there exists a normal subgroup $N$ of $W$ with $\Phi(W) < N \leqslant C$ such that $W/N$ is a cyclic $3$-group, $N/\Phi(W)$ is a nonabelian chief factor of $W$, isomorphic to $S^m$ for some nonabelian simple group $S$, and $\Phi(W)$ is a $2$-group containing $K$. Writing $|\Phi(W)|=2^a$, the integer $a$ is even.
$$\xymatrix{\{1\} \ar@{-}[rr]_{2^a} & & \Phi(W) \ar@{-}[rr]_{S^m} & & N \ar@{-}[rr]_{3^b} & & C \ar@{-}[rr]_{3} & & W}$$

\begin{proof}
Let $N$ be a normal subgroup of $W$ such that $W/N$ is cyclic and $N$ is of minimal order with this property. If $W/N$ is not a $3$-group, then there exists $L/N \unlhd W/N$ such that $|W:L|$ is a prime distinct from $3$, so that $LC=W$ (being $|W:C|=3$) hence $L$ is solvable. Since $W/L$ is solvable, this contradicts the fact that $W$ is nonsolvable. So $W/N$ is a cyclic $3$-group. Moreover $N$ is contained in $C$ because otherwise $NC=W$ and then $N$ would be solvable, so $W$ would be solvable as well. Consider a normal subgroup $R$ of $W$ contained in $N$ with the property that $N/R$ is a minimal normal subgroup of $W/R$. We claim that $R=\Phi(W)$. First, note that $\Phi(W)$ is a $2$-group because otherwise there would exist a nontrivial Sylow subgroup $P$ of $\Phi(W)$ of odd order, and since $P \unlhd_c \Phi(W) \unlhd W$, being $\Phi(W)$ nilpotent, $P$ would be a nontrivial normal subgroup of $W$ of odd order, contradicting $O(W)=\{1\}$. Since $\Phi(W)/\Phi(W) \cap N \cong \Phi(W)N/N \leqslant \Phi(W/N)$ is a $3$-group, we deduce that $\Phi(W) \leqslant N$. By minimality of $N$, $W/R$ is not a cyclic $3$-group, so if $x \in W$ is a $3$-element that generates $W/C$, then $R \langle x \rangle$ is a proper subgroup of $W$, moreover $R \langle x \rangle C = W$, so $R$ is solvable. Since $R \leqslant N \leqslant C$, to prove that $R \leqslant \Phi(W)$ it is enough to prove that if $M$ is a maximal subgroups of $W$ distinct from $C$, then $R \leqslant M$. We have $MC=W$, so $M$ is solvable hence $MR \neq W$ being $W$ nonsolvable and $M,R$ solvable. Therefore $R \leqslant M$, implying that $R \leqslant \Phi(W)$. Since $W/N$ is a $3$-group and $W$ is nonsolvable, we deduce that $\Phi(W) \neq N$, in other words $\Phi(W)$ is properly contained in $N$. Since $R \leqslant \Phi(W) < N$ and $N/R$ is a minimal normal subgroup of $W/R$, we deduce that $R=\Phi(W)$. Since $\Phi(W)$ is nilpotent, the nonsolvability of $W$ implies that $N/\Phi(W)$ is a nonabelian chief factor of $W$, isomorphic to a direct power $S^m$ for some nonabelian simple group $S$. Observe that $S$ has more than $3$ involutions: indeed, no involution of $S$ is central, being $S$ simple, and if $S$ had three involutions then the kernel of the conjugation action of $S$ on them would be a normal subgroup of $S$ of index at most $3!=6$, a contradiction. We deduce that $\Phi(W) \neq \{1\}$, so $\Phi(W)$, being a $2$-group, contains at least one involution, hence it contains all $3$ of them being $\Phi(W) \unlhd W$. Writing $|\Phi(W)|=2^a$, the integer $a$ is even by Lemma \ref{propertiesKTS}(v).
\end{proof}

\subsection{Step 2} \label{corephi} Recall that a finite group $H$ is said to be primitive if it admits a maximal subgroup $M$ with trivial normal core, in other words the intersection of the conjugates of $M$ in $H$, usually denoted $M_H$, is trivial. Note that, if $M$ is any maximal subgroup of a finite group $H$, then $H/M_H$ is primitive since $M/M_H$ is a maximal subgroup of $H/M_H$ with trivial normal core. Let $M$ be a maximal subgroup of $W$ with $M \neq C$. We will prove that the normal core $M_W$ of $M$ in $W$ equals $\Phi(W)$ and consequently $W/\Phi(W)$ is a primitive group.

\begin{proof}
By Lemma \ref{propertiesKTS}(i), the maximal subgroups of $W$ distinct from $C$ are solvable by minimality of $W$, since they supplement $C$. If $M_1$ is a maximal subgroup of $W$ such that $M_1\neq C$ then $M_W \leqslant M_1$, since otherwise we would have
$M_1M_W=W$, contradicting the fact that $W$ is nonsolvable, being $M_1$ and $M_W$ solvable. So
every maximal subgroup of $W$ different from $C$ contains $M_W$, implying that $\Phi(W) = M_W\cap C$. To conclude the proof, we need to prove that $M_W \leqslant C$. If this is not the case, then $C M_W = W$.
Note that $W/M_W=CM_W/M_W\cong C/\Phi(W)$ has a subgroup $N/\Phi(W)\cong S^m$. Since every nonabelian simple group has order divisible by $4$, there exist two subgroups $A, B\leqslant W$ containing $M_W$
such that $A\leqslant B$ and $|B:A|=|A:M_W|=2$. Since $A$ and $B$ contain $M_W$, $AC=BC=W$, hence
$A$ and $B$ are both $3$-pyramidal by Lemma \ref{propertiesKTS}(i). Since $K \leqslant \Phi(W) \leqslant M_W \leqslant A$ and $|K|=4$, this contradicts Lemma \ref{propertiesKTS}(iv).
\end{proof}

\subsection{Step 3} \label{supplementsocle} Let $G:=W/\Phi(W)$, it is a primitive group whose socle is $N/\Phi(W) \cong S^m$, it is the unique minimal normal subgroup of $G$. Let $T_1$ be the first direct factor of $N/\Phi(W) \cong S^m$ and let $G_1 := N_G(T_1)/C_G(T_1)$. Then $G_1$ is an almost-simple group with socle $T_1 C_G(T_1)/C_G(T_1) \cong S$ and we will identify $G_1$ with a subgroup of $\Aut(S)$ containing $S$. We say that a finite group is of \textit{even type} if, in the prime factorization of its order, the prime $2$ appears with even multiplicity. Lemma \ref{propertiesKTS}(iv-v) implies that $G$ is of even type. We will prove that, if $U$ is a maximal subgroup of $G_1$ such that $US=G_1$, then $U$ is a solvable group of even type.

\begin{proof}
We refer to \cite[Chapter 1]{BE} for the general properties of primitive groups. Since $N/\Phi(W)$ is a nonabelian minimal normal subgroup of $G$ and $W/N$ is cyclic, $G$ is a primitive group of type $2$, meaning that it admits a unique minimal normal subgroup which is nonabelian. Since $W/N$ is a cyclic $3$-group, we have an embedding of $G$ in the standard wreath product $G_1 \wr P$ where $P$ is a cyclic $3$-subgroup of $\Sym(m)$, the image of the permutation representation of $G$ acting transitively by conjugation on the $m$ minimal normal subgroups of its socle (see \cite[Remark 1.1.40.13]{BE}). In particular, $m$ is a power of $3$.

Let $V := (U \cap S)^m \leqslant S^m = \soc(G)$. We claim that $N_G(V) \soc(G) = G$. Let $g = (x_1,\ldots, x_m) \gamma\in G$, where $x_i \in G_1$ for all $i$ and $\gamma\in P$. Since $US=G_1$, there exist $s_i \in S$, $t_i\in U$ such that $x_i = s_it_i$ for all $i=1,\ldots,m$, and setting $n:=(s_1,\ldots, s_m) \in \soc(G)$ we have $h := (t_1,\ldots, t_m)\gamma = n^{-1} g \in G$, therefore
$$V^h = ((U \cap S)^{t_1} \times \cdots \times (U \cap S)^{t_m})^\gamma = V^\gamma = V.$$
We deduce that $h=n^{-1}g\in N_G(V)$ and since $n \in \soc(G)$ the claim follows.

We claim that $N_G(V) \neq G$. If this was not the case, then $V$ would be normal in $G$ and being $N$ a minimal normal subgroup of $G$, and $U \cap S \neq S$ being $US=G_1$, since $S$ is simple we deduce that $U \cap S=\{1\}$ and $G_1/S \cong U$. It follows that $U$ is isomorphic to a section of $G/\soc(G) \cong W/N$, which is cyclic, therefore $U$ is a cyclic maximal subgroup of the almost-simple group $G_1$, contradicting Herstein's theorem, which says that a nonsolvable finite group cannot have abelian maximal subgroups (see \cite[Theorem 5.53]{AM}). So the claim follows.

Since $\Phi(W)$ is a nontrivial $2$-group, the preimage of $N_G(V)$ in $W$ is a proper subgroup of $W$ of even order and it is not contained in $C$, so it is $3$-pyramidal by Lemma \ref{propertiesKTS}(i) hence it is solvable by minimality of $W$. So $U \cap S$ is solvable. Since $U/U \cap S \cong US/S = G_1/S$ is a cyclic $3$-group, this implies that $U$ is solvable.

Summarizing, we have $N_G(V) \soc(G) = G$, $N_G(V) \neq G$, $G_1/S$ is a cyclic $3$-group and $G_1$ has only one nonsolvable maximal subgroup, which contains $S$. By Lemma \ref{G0=S}, $U \cap S$ is a maximal subgroup of $S$. We claim that $N_G(V) \cap \soc(G) = V$. To show this, let $L$ be a maximal subgroup of $G$ containing $N_G(V)$. Since $V \leqslant N_G(V) \leqslant L$ and $N_G(V) \soc(G)=G$, the maximal subgroup $L$ is either of product type or of diagonal type (the two types described in \cite[Remark 1.1.40.19]{BE}). But since $L \soc(G)=G$, $L$ is solvable, and hence $L$ must be of product type.
Thus we can write $L \cap \soc(G) = H_1 \times \cdots\times H_m$ with $H_i < S$ for all $i$ and the $H_i$'s are pairwise isomorphic. We have $L \cap \soc(G) \neq \{1\}$ since otherwise $L$ would be isomorphic to $G/\soc(G)$, so it would be a cyclic maximal subgroup of $G$. Therefore $H_i \neq \{1\}$ for all $i$. Since $U \cap S$ is a maximal subgroup of $S$ and $V \leqslant N_G(V) \cap \soc(G) \leqslant L \cap \soc(G)$, we deduce that $H_i = U \cap S$ for all $i$, hence $L \cap \soc(G) = V$. On the other hand $L \leqslant N_G(L \cap \soc(G)) = N_G(V)$ hence $L = N_G(V)$ by maximality of $L$, therefore $N_G(V) \cap \soc(G) = V$. 

Since $N_G(V) \soc(G) = G$ and $U/U \cap S \cong US/S = G_1/S \cong C_{3^s}$ is a cyclic $3$-group, writing $|U|=2^d \cdot r$ with $r$ odd and $G/\soc(G) \cong C_{3^t}$, we have 
$$|N_G(V)| = |G:\soc(G)| \cdot |N_G(V) \cap \soc(G)| = 3^t \cdot |U \cap S|^m = 3^{t-sm} \cdot 2^{dm} \cdot r^m.$$ 
Since the preimage of $N_G(V)$ in $W$ is $3$-pyramidal by Lemma \ref{propertiesKTS}(i), and $\Phi(W)$ is of even type by Lemma \ref{propertiesKTS}(v), $N_G(V)$ is of even type by Lemma \ref{propertiesKTS}(iv), in other words $dm$ is even. Since $m$ is a power of $3$, we conclude that $d$ is even.
\end{proof}

\subsection{Step 4} \label{G1=S} $G_1 = S$.

\begin{proof}
Let $M$ be any maximal subgroup of $W$ different from $C$, so that $M$ is solvable, $\Phi(W)=M_W$ and $H := M/\Phi(W)$ is a point stabilizer of the primitive group $G$. Since $H$ is solvable, it is of product type, so $H$ is conjugate in $G$ to $N_G((H_1 \cap S)^m)$ for a suitable maximal subgroup $H_1$ of $G_1$ such that $H_1S=G_1$. By \cite[Theorem 1.1]{LZ}, $G_1$ has a normal subgroup $G_0$ which is minimal such that $H_0 := H_1 \cap G_0$ is maximal in $G_0$ and the pair $(G_0,H_0)$ is explicitly listed in \cite[Tables 14-20]{LZ} up to conjugacy. Lemma \ref{G0=S} implies that $G_0=S$.

We have $S \leqslant G_1\leqslant \Aut(S)$ and $G_1/S$ is a cyclic $3$-group. In our discussion, we will consider suitable nonsolvable maximal subgroups $R$ of $S$ and we will apply Lemma \ref{X-ord} with $R$ in place of $H$ and $G_1$ in place of $X$. Since $G_1/S$ is a cyclic $3$-group, if $c \leqslant 2$ then $R$ is $G_1$-ordinary and hence $N_{G_1}(R)$ is a nonsolvable maximal subgroup of $G_1$ supplementing $S$, this contradicts Step 3. Note that our ``$c$'' is the same as in \cite[Paragraph 3.2]{KL} and \cite[Chapter 8]{BHR}.

Assume $G_1 \neq S$ by contradiction. In the following discussion we will use \cite[Tables 14-20]{LZ} recalling that $G_0=S$ in our situation. Since $G_1/S$ is a $3$-group, if $\Out(S)$ is a $2$-group we immediately have a contradiction. This is the case if $S$ is an alternating or sporadic group. If $S$ is a symplectic group, then by \cite[Table 17]{LZ} again $\Out(S)$ is a $2$-group, a contradiction. In the following discussion, the existence of a nonsolvable $G_1$-ordinary maximal subgroup of $S$ gives a contradiction by Lemma \ref{X-ord}, as explained in the previous paragraph.

If $S$ is a unitary group, $S \cong \PSU_n(q)$ where $q=p^f$ and
$n\geqslant 4$, then by \cite[Tables 8.10, 8.20, 8.26, 8.37, 8.46, 8.56, 8.62, 8.72, 8.78]{BHR} and \cite[Table 3.5.B]{KL}, $S$ has an ordinary parabolic maximal subgroup of type $[q^{2n-3}]:\SU_{n-2}(q):(q^2-1)/d$, with $d=(n,q+1)$, which is nonsolvable unless $(n,q) \in \{(4,2),(4,3),(5,2)\}$, in which case $\Out(S)$ is a $2$-group.
Assume $n=3$, so that $q \geqslant 4$ since $\PSU_3(2)$ is not a simple group and $\Out(\PSU_3(3)) \cong C_2$.
By \cite[Table 8.5]{BHR} $S$ has a nonsolvable ordinary maximal subgroup of type $\GU_2(q)/d$,
where $d=(q+1,3)$. If $n=2$ then $S \cong \PSL_2(q)$ will be considered below.

If $S$ is an orthogonal group, by \cite[Table 19]{LZ} $S$ is isomorphic to $\OO_7(3)$, $\OO_8^+(q)$, $\OO_{12}^+(3)$ or $\OO_{16}^+(3)$. Since $\Out(S)$ is a $2$-group unless $S\cong \OO_8^+(q)$, we assume this is the case. If $q=2$ then $\Out(S)\cong S_3$ and if $q=3$ then $\Out(S) \cong S_4$, hence $G_1 = S.\langle \tau\rangle$ where $\tau$ is the triality automorphism in these two cases, and $G_1$ has a nonsolvable maximal subgroup of type $\PSU_3(3):2\times 3$ if $q=2$ and $G_2(3)\times 3$ if $q=3$, not containing $S$ (see \cite{CCNPW}), contradicting Step 3.
If $q \geqslant 4$, then $S$ has a nonsolvable ordinary maximal subgroup
of type $[q^9](\frac{1}{d}\GL_2(q)\times\Omega_4^+(q)).d$ where $d=(q-1,2)$ by \cite[Table 8.50]{BHR} and \cite[Table I]{Kleidman}.

Assume $S$ is a linear group, $S\cong \PSL_n(q)$, where $n \geqslant 2$ and $q=p^f$.
If $n\geqslant 3$ then $S$ has a maximal parabolic subgroup
$P_1$ of type $[q^{n-1}].\SL_{n-1}(q).(q-1)/d$, $d=(q-1, n)$. By
\cite[Proposition 4.1.17]{KL} and Lemma \ref{X-ord}, $P_1$ is $G_1$-ordinary. It is nonsolvable unless $n=3$ and $q \leqslant 3$, in which case $\Out(S) \cong C_2$ and $G_1=S$. Assume $n=2$, so that $S \cong \PSL_2(q)$. Since $\Out(S) \cong C_{(2, q-1)}\times C_f$ and $G_1/S$ is a cyclic $3$-group, $3$ divides $f$. By \cite[Table 8.1]{BHR}, $S$ has an ordinary maximal subgroup of type $\PSL_2(q_0)$, where $q_0=p^{f/3}=\sqrt[3]{q}$, if $q_0 \neq 2$, and this maximal subgroup is nonsolvable if $q_0 \geqslant 4$. If this is not the case, then $q_0 \leqslant 3$ i.e. $q \in \{8,27\}$.
Thus $G_1/S \cong C_3$ and $G_1$ has a maximal subgroup of type $7:6$ or $13:6$ supplementing $S$ (see \cite{CCNPW}), contradicting Step 3, since such subgroups are not of even type.

We now consider exceptional groups. The group $S$ appears in the lines of \cite[Table 20]{LZ} in which $G_0$ is a simple group. We first suppose that $S$ is isomorphic to a Suzuki group $\Sz(2^m) = {^2}B_2(2^m)$, with $m$ odd. If $m$ is not prime, then we can write $m=ab$ with $b$ prime and $a>2$, so by \cite[Table 8.16]{BHR}, $S$ has
a nonsolvable ordinary maximal subgroup isomorphic to $\Sz(2^a)$.
If $m$ is a prime then, since $3$ divides $|\Out(S)|=m$,
we have $m=3$, $S \cong \Sz(8)$ and $G_1 = S.C_3$ has a maximal subgroup of type $7:6$ supplementing $S$ (see \cite{CCNPW}). This contradicts Step 3, since such subgroup is not of even type. We use the notation $\Ree(q)={^2}G_2(q)$, as in \cite{LZ}. 
Observe that $\Ree(3)'\cong \PSL_2(8)$ was considered in the discussion of linear groups and
$\Out(S)\cong C_2$ if $S\cong G_2(2)'$, $G_2(3)$, ${}^2F_4(2)'$ or $F_4(2)$, so we obtain a contradiction in these cases. By \cite[Main Theorem, Tables 5.1, 5.2]{MJG}, taking $L=S$ and $X=G_1$, if $S$ is one of the exceptional groups $\Ree({q})$ for $q\geqslant 27$, ${^2}F_4(q)$ for $q\geqslant 8$, ${^3}D_4(q)$, ${^2}E_6(q)$, $E_6(q)$ and $E_8(q)$, then $G_1$ has a maximal subgroup $M=N_X(D)$ which is nonsolvable because $N_{G_{\sigma}}(D)$ is nonsolvable and its derived subgroup is contained in $M$, indeed $(G_{\sigma})' = L \leqslant X$
and hence $(N_{G_{\sigma}}(D))' \leqslant N_{G_{\sigma}}(D)\cap X\leqslant N_X(D)=M$. Moreover, $M$ does not have $S$ as a composition factor, therefore $MS=G_1$. This gives a contradiction.
\end{proof}

\subsection{Step 5} We conclude that $G_1=S$. Since $G$ is a subgroup of $S\wr P$ containing $S^m = \soc(G)$, with $G$ projecting
surjectively onto the transitive cyclic group $P \leqslant \Sym(m)$, we deduce that $G$ is isomorphic to the standard wreath product $S\wr P$, where $P$ acts on $S^m$ by permuting the coordinates. Consider
$$\Delta=\{(s, s,\cdots, s)\ :\ s\in S\} \leqslant S^m, \hspace{1cm} H:=N_G(\Delta).$$
It is clear that $P \leqslant H$, hence $\soc(G)H=G$. Since $G$ has a normal subgroup of index $3$, we have $m\geqslant 2$, therefore $\Delta$ is not normal in $G$, being $\soc(G)$ a minimal normal subgroup of $G$, therefore $H \neq G$. Since $\Delta \cong S$ is nonsolvable, we obtain that the preimage of $H$ in $W$ is a nonsolvable proper subgroup of $W$, of even order and supplementing $C$, so it is $3$-pyramidal by Lemma \ref{propertiesKTS}(i). This contradicts the minimality of $W$.

This concludes the proof of the fact that $3$-pyramidal groups are solvable. \hfill$\square$

\section{Proof of the Main Theorem} \label{proofmain}

In this section, we will study the structure of $3$-pyramidal groups and prove our Main Theorem.  
Recall that, if $G$ is a $p$-solvable group, meaning that it admits a series of normal subgroups $\{1\} = V_0 < \ldots < V_n = G$
such that each factor $V_{i+1}/V_i$ is either a $p$-group or a $p'$-group, the $p$-length of $G$ is the length $l$ of the upper $p$-series
$$\{1\} = P_0 \leqslant N_0 < P_1 < N_1 < P_2 < \ldots < P_l \leqslant N_l = G$$
where $N_k/P_k$ is the largest normal $p'$-subgroup of
$G/P_k$ and $P_{k+1}/N_k$ the largest normal $p$-subgroup of $G/N_k$, for $k=0,\ldots,l-1$.

Recall that a group $G$ is called homocyclic if it is a direct product of isomorphic cyclic groups, and $G$ is called a Suzuki $2$-group if $G$ is a nonabelian $2$-group with more than one involution and there exists a solvable subgroup of $\Aut(G)$ which permutes the set of involutions of $G$ transitively. Since $3$-pyramidal groups are solvable, Thompson's result \cite[Theorem 8.6 of Chapter IX]{HB} is now very useful for us. It states the following.

\begin{thm}[Thompson] \label{Thompson}
Suppose that $G$ is a solvable group of even order and that the Sylow $2$-subgroup of $G$ contains more than one
involution. Suppose that all the involutions in $G$ are conjugate. Then the $2$-length of
$G$ is $1$ and the Sylow $2$-subgroups of $G$ are either homocyclic or Suzuki $2$-groups.
\end{thm}

Let $G$ be a $3$-pyramidal group and write $|G| = 2^n \cdot d$, where $d$ is odd. Let $K$ be the subgroup of $G$ generated by the three involutions $i,j,k$ and let $C:=C_G(K)$. If $n=1$ then the result follows from Lemma \ref{propertiesKTS}(iii). Now assume that $n\geqslant 2$.
Lemma \ref{propertiesKTS}(iv) implies that $n$ is even, $K\cong C_2\times C_2$ and $G/C\cong A_3$. Since $K$ and $O(G)$ are normal subgroups of $G$ of coprime orders, $K\cap O(G) = \{1\}$ hence $O(G) \leqslant C$. 
By Lemma \ref{propertiesKTS}(ii) we may assume that $O(G)=\{1\}$. By Section \ref{KTSsolvable}, $G$ is solvable and Theorem \ref{Thompson} applies, so the $2$-length of $G$ is $1$. Since $O(G)=\{1\}$, this means that the Sylow $2$-subgroup $N$ of $G$ is normal hence it has a complement $A$ in $G$ by the Schur-Zassenhaus Theorem, i.e., $G=N\rtimes A$. Since $C_A(N)$ is a normal subgroup of $G$ of odd order, it is trivial, hence $C_G(N) = Z(N)$ is a $2$-group and $A$ is isomorphic to a subgroup of $\Aut(N)$.
In the light of Theorem \ref{Thompson}, $N$ is either homocyclic or a Suzuki $2$-group. 

Assume that $N$ is a homocyclic group, i.e. a direct product of pairwise isomorphic cyclic groups.
Since $N$ has three involutions, $N\cong C_{2^m}\times C_{2^m}$ for some positive integer $m$. It is easy to see that $\Aut(N)$ is isomorphic to the group of $2 \times 2$ matrices with coefficients in $\mathbb{Z}/2^m\mathbb{Z}$ and invertible determinant, therefore $|\Aut(N)|=3 \cdot 2^{4m-3}$.
By Sylow's theorem, $A$ is conjugate in $\Aut(N)$ to $\langle \gamma \rangle$ where $\gamma: N \to N$ is defined by $(a,b) \mapsto (b, (ab)^{-1})$. Therefore $G \cong N \rtimes \langle \gamma \rangle$.

We next assume that $N$ is a Suzuki $2$-group. Since $N$ has three involutions, we have $\Phi(N)=N'=Z(N) \cong C_2^2$ by \cite[Lemma 7.8, Theorem 7.9 of Chapter VIII]{HB}. A simple inspection using \cite{GAP} shows that no group of order $16$ has these properties, therefore \cite{GH} implies that $|N|=64$ and $N$ is the unique Suzuki $2$-group of order $64$, SmallGroup(64, 245), $|\Aut(N)|=3\cdot 5\cdot2^{10}$ and $G$ is isomorphic to $N \rtimes A$ where $A$ is a subgroup of $\Aut(N)$ of order $3$ or $15$. Since the subgroups of $\Aut(N)$ of order $3$ or $15$ are Hall subgroups and $\Aut(N)$ is solvable, they form a unique conjugacy class of subgroups, so $G$ is determined completely up to isomorphism.

Conversely, it is not difficult to see that $G$ is $3$-pyramidal if
$G \cong S_3 \times H$ where $H$ is a group of odd order. Assume now that $4$ divides $|G|$, let $O:=O(G)$ and assume that $O$ centralizes all involutions of $G$ and that $G/O$ is one of the groups in cases (2) and (3). Since $4$ divides $|G/O|$ and $G/O$ is $3$-pyramidal, it has exactly $3$ involutions $iO$, $jO$ and $kO$. 
The union $iO \cup jO \cup kO$ contains all involutions of $G$. Since $O$ is a group of odd order centralizing all involutions of $G$, each of $iO$, $jO$, $kO$ contains a unique involution,
therefore $G$ has precisely three involutions, which we may assume to be $i$, $j$, $k$. Since $iO$, $jO$ and $kO$ are conjugate in $G/O$, there exists $g \in G$ such that $i^g O = (iO)^{gO} = jO$. Since $j$ is the unique involution in $jO$, we obtain that $i^g=j$. The same argument shows that $j$ and $k$ are conjugate in $G$, so $G$ is $3$-pyramidal.
\hfill$\square$


\begin{thebibliography}{99}

\bibitem{SMG} S. Bonvicini, M. Buratti, M. Garonzi, G. Rinaldi, and T. Traetta. \textit{The first families of highly symmetric kirkman triple systems whose orders fill a congruence class.} Designs, Codes and Cryptography, 89 (2021) 2725--2757.

\bibitem{HB} { B. Huppert, N. Blackburn.\textit{ Finite groups II.} Berlin: Springer-Verlag, 1982.}

\bibitem{shaw} {D. L. Shaw. \textit{The sylow 2-subgroups of finite, soluble groups with a single class of involutions.} Journal of Algebra, 16:14--26, 1970.}

\bibitem{LZ} {C. H. Li, H. Zhang.\textit{ The finite primitive groups with soluble stabilizers, and the edge-primitive s-arc transitive graphs.} Proceedings of the London Mathematical Society, (3) 103 (2011) 411--472.}

\bibitem{CCNPW} {J. H. Conway, R. T. Curtis, S. P. Norton, R. A. Parker and R. A. Wilson. \textit{Atlas of finite groups.} Oxford University Press, London, 1985.}

\bibitem{BHR} {J. Bray, D. Holt, and  C. Roney-Dougal.\textit{The maximal subgroups of the low-dimensional finite classical groups.} Cambridge University Press, London, 2013.}

\bibitem{Kleidman} P. B. Kleidman. \textit{The maximal subgroups of the finite 8-dimensional orthogonal groups $\OO_8(q)$ and of their automorphism groups.} Journal of Algebra, 110, 1987.

\bibitem{DJS} {D. J. S. Robinson.\textit{ A course in the theory of groups.} Berlin: Springer, 1982.}

\bibitem{BE} {A. Ballester-Bolinches and L. M. Ezquerro. \textit{Classes of Finite Groups.} Springer, 2006.}

\bibitem{AM} {A. Machi. \textit{ Groups: An Introduction to Ideas and Methods of the Theory of Groups.}  Springer, 2012.}

\bibitem{KL} {P. B. Kleidman and M. W. Liebeck.\textit{The subgroup structure of the finite classical groups.} Cambridge University Press, London, 1990.}

\bibitem{MJG} {M. W. Liebeck, J. Saxl, G. M. Seitz.\textit{ Subgroups of maximal rank in finite exceptional groups of Lie type.} Proceedings of the London Mathematical Society, 65 (1992) 297-325.}

\bibitem{GAP} {The GAP Group. \textit{GAP-Groups, Algorithms, and Programming.} Version 4.7.8.}

\bibitem{GH} {G. Higman, \textit{ Suzuki 2-groups.} Illinois Journal of Mathematics, 7 (1963) 79-96.}
\end{thebibliography}
\end{document}